\documentclass{article}
\pdfoutput=1
\usepackage[english]{babel}
\usepackage{amsmath}
\usepackage{amsthm}
\usepackage{amssymb}
\usepackage{mathtools}
\usepackage{enumitem}
\usepackage{hyperref}
\usepackage{color}
\newtheorem{theo}{Theorem}
\newtheorem{defi}{Definition}
\newtheorem{lemma}{Lemma}

\newtheorem{inva}{Invariant}
\newtheorem{cor}{Corollary}

\title{The random strategy in Maker-Breaker graph minor games}
\date{\vspace{-5ex}}
\author{Ander Lamaison\thanks{Funded by the Deutsche Forschungsgemeinschaft (DFG, German Research Foundation) under Germany's Excellence Strategy - The Berlin Mathematics Research Center MATH+ (EXC-2046/1, project ID: 390685689).}}

\begin{document}
\maketitle

\begin{abstract}

In a $(1:b)$ biased Maker-Breaker game, how good a strategy is for a player can be measured by the bias range for which its rival can win, choosing an appropriate counterstrategy. Bednarska and \L{}uczak proved that, in the $H$-subgraph game, the uniformly random strategy for Maker is essentially optimal with high probability. Here we prove an analogous result for the $H$-graph minor game, and we study for which choices of $H$ the random strategy is within a factor of $1+o(1)$ of being optimal.
\end{abstract}

\section{Introduction}

A Maker-Breaker game is defined by a set $X$, called the board, and a family $\mathcal{F}$ of subsets of $X$, called the winning sets. Two players, called Maker and Breaker, claim elements of the board, in rounds. If the game has bias $(1:b)$, in every round Maker claims one element and Breaker claims $b$ elements. Once an element has been claimed, it cannot be claimed again. Maker wins if, after the board is completely claimed, some winning set has been fully claimed by Maker, otherwise Breaker wins.

We will consider that $X$ is the edge set of a complete graph $K_n$. Instead of a family $\mathcal{F}$, we will consider a monotone increasing graph property $\mathcal{P}$. A graph property is a family of graphs $\mathcal{P}=\cup_{i=0}^\infty\mathcal{P}_i$, where the graphs in $\mathcal{P}_n$ each have $n$ vertices. We say that $\mathcal{P}$ is monotone increasing if, for every $n$, if $G_1$ and $G_2$ are two graphs on $n$ vertices and $G_1\subseteq G_2$, then $G_1\in\mathcal{P}$ implies $G_2\in\mathcal{P}$. We denote by $M$ the graph formed by Maker's edges at the end of the game. In the game on $n$ vertices defined by $\mathcal{P}$, Maker wins if $M$ is in $\mathcal{P}_n$. We will be interested in the behavior of the game as $n$ goes to infinity for a fixed graph property $\mathcal{P}$.

We define the threshold bias of the game, denoted by $b^*=b^*(n,\mathcal{P})$, as the minimum value such that Breaker has a winning strategy in the $(1:b^*)$ biased game. We will omit one or both parameters if the number of vertices or the graph property defining the game are clear.

A possible strategy for Maker is to play randomly in every round. That is, every time that Maker needs to claim an edge, the choice is made uniformly at random from the set of unclaimed edges. A similar strategy can be used by Breaker. We will add the prefix {\tt Random}- to the name of a player to denote that it is following the random strategy, or {\tt Clever}- to indicate that it is following the strategy that maximizes its chances of winning.

To be precise, when both players are {\tt Clever}- the player who has a winning strategy always wins. When only one of the players is {\tt Clever}- (say {\tt RandomMaker} versus {\tt CleverBreaker}), for any $\mathcal{P}$, $n$ and $b$, each deterministic strategy by {\tt CleverBreaker} has a probability of winning against Maker's random strategy. {\tt CleverBreaker} follows the deterministic strategy that maximizes said probability.

If one or both players follow the random strategy, the winner of the game is not defined deterministically. Instead, each player has a probability of winning. We change the definition of threshold to account for this:

\begin{defi}
Let $X,Y\in\{{\tt Random-},{\tt Clever-}\}$ (or $\{R,C\}$ for short). Let $\mathcal{P}$ be a monotone increasing graph property. We say that a function $f(n)$ is a threshold (for $\mathcal{P}$) if:
\begin{itemize}
\item $\Pr\left({\tt XMaker}\text{ wins against }{\tt YBreaker}\right)\rightarrow 1$ for $b=o(f(n))$.
\item $\Pr\left({\tt XMaker}\text{ wins against }{\tt YBreaker}\right)\rightarrow 0$ for $b=\omega(f(n))$.
\end{itemize}
Similarly, we say that $f(n)$ is a sharp threshold (for $\mathcal{P}$) if for every $\epsilon>0$ we have 
\begin{itemize}
\item $\Pr\left({\tt XMaker}\text{ wins against }{\tt YBreaker}\right)\rightarrow 1$ for $b=(1-\epsilon)f(n)$.
\item $\Pr\left({\tt XMaker}\text{ wins against }{\tt YBreaker}\right)\rightarrow 0$ for $b=(1+\epsilon)f(n)$.
\end{itemize}
\end{defi}

We will sometimes use the notation $b^*_{XY}\sim f(n)$ or $b^*_{XY}\approx f(n)$ as shorthand to indicate that the function $f(n)$ is a threshold or a sharp threshold, respectively, in the {\tt XMaker}-{\tt YBreaker} game. We will write $b^*_{XY}\sim b^*_{X'Y'}$ if there exists a function with $b^*_{XY}\sim f(n)\sim b^*_{X'Y'}$, and we will write similarly $b^*_{XY}\approx b^*_{X'Y'}$ if a common sharp threshold exists. It is important to note we do not treat $b^*_{XY}$ as a function, and $b^*_{XY}$ will only be used to define the asymptotic behavior of the thresholds: if we have $f_1\sim b^*_{XY}\sim f_2$ and $g_1\approx b^*_{XY}\approx g_2$, we have $f_2=\Theta(f_1)$ and $g_2=(1+o(1))g_1$.

When comparing the different thresholds that are obtained in this way, two phenomena have been observed to arise in different games. For some properties $\mathcal{P}$, we have $b^*_{CC}\sim b^*_{RR}$ or $b^*_{CC}\approx b^*_{RR}$, meaning that two clever players get roughly the same advantage from a certain bias as two random players. This phenomenon is called \textit{probabilistic intuition}. If both players play randomly, the resulting graph $M$ is the Erd\H{o}s-R\'enyi random graph with a prescribed number of edges, $G\left(n,\left\lceil\frac{{n \choose 2}}{b+1}\right\rceil\right)$. Probabilistic intuition holds for example in the Hamiltonicity game \cite{Kri} (Maker wins if $M$ is Hamiltonian, $b^*_{CC}\approx \frac{n}{\log n}\approx b^*_{RR}$), the connectivity game \cite{GebSza} (Maker wins if $M$ is connected, $b^*_{CC}\approx \frac{n}{\log n}\approx b^*_{RR}$), and, in the weaker sense, the non-planarity game \cite{HKSS} (Maker wins if $M$ is non-planar, $b^*_{CC}\approx \frac n2\sim b^*_{RR}$).

The other phenomenon concerns the case when $b^*_{CC}\sim b^*_{RC}$ or $b^*_{CC}\approx b^*_{RC}$, where intuitively the random strategy for Maker is close to being optimal. For example, given a fixed graph $H$, for the $H$-subgraph game (Maker wins if $H$ is a subgraph of $M$) Bednarska and \L{}uczak \cite{BedLuc} proved that $b^*_{CC}\sim n^{1/m_2(H)}\sim b^*_{RC}$, where \[m_2(H)=\min\limits_{\substack{H'\subseteq H\\ v(H')\geq 3}}\frac{e(H')-1}{v(H')-2}.\] This result was generalized in \cite{KRSS} to the hypergraph setting.

A graph $G$ is said to have $H$ as a minor if there exists a family of pairwise disjoint and non-empty sets $S_v\subseteq V(G)$ for each $v\in V(H)$ such that $G[S_v]$ is connected for every $v$, and whenever $vw\in E(H)$ there exists an edge in $G$ between $S_v$ and $S_w$. In the $H$-minor game, Maker's goal is to ensure that $H$ is a minor of $M$. In other words, this is the Maker-Breaker game where $\mathcal{P}$ is the family of graphs having $H$ as a minor.

We will prove the optimality of the random strategy in this game, in the weak sense of $b^*_{CC}\sim b^*_{RC}$. Moreover, for every (fixed) choice of $H$ we will find a sharp threshold for the {\tt CleverMaker}-{\tt CleverBreaker} game, and prove that in many cases it is sharply matched in the {\tt RandomMaker}-{\tt CleverBreaker} game:

\begin{theo}\label{main} Let $H$ be a fixed graph, and let $\tau(H)$ denote the maximum number of edges in a component of $H$. In the $H$-minor game: \begin{enumerate}[label=(\roman*)]\item $b^*_{CC}\approx \frac{n^2}{2e(H)-2}\approx b^*_{RC}$, if $\tau(H)=1$.\item $b^*_{CC}\approx 2n\approx b^*_{RC}$, if $\tau(H)=2$.\item $b^*_{CC}\approx n\sim b^*_{RC}$, if $H$ is a forest and $\tau(H)\geq 3$.\item $b^*_{CC}\approx \frac n2\approx b^*_{RC}$, if $H$ is not a forest.\end{enumerate}\end{theo}

Note that, if $H$ is a forest with $\tau(H)\geq 3$, then we only claim $b^*_{RC}\sim n$. This is not an artifact of the proof: there are choices of $H$ where $b^*_{RC}\not\approx n$, with the path on eleven vertices being an explicit example.

\begin{theo}\label{pathrmcb} Let $P_{11}$ be a path on eleven vertices. In the $P_{11}$-minor game with bias $b=0.99 n$, {\tt CleverBreaker} wins with high probability against {\tt RandomMaker}.\end{theo}

A consequence of this result is that the theorem of Bednarska and \L{}uczak which shows that $b^*_{CC}\sim b^*_{RC}$ in the $H$-subgraph game cannot always be strengthened to $b^*_{CC}\approx b^*_{RC}$:

\begin{cor}\label{bedlucnot} Let $H$ be a path with at least ten edges or a three-legged spider (three disjoint paths meeting at a common endpoint) containing a path with ten edges as a subgraph. Then, in the $H$-subgraph game, $b^*_{CC}\approx n\not\approx b^*_{RC}$.\end{cor}

A related notion is that of a topological minor. A graph $H'$ is a subdivision of $H$ if it can be obtained from $H$ by repeatedly applying the following procedure: take an edge $uv$ in the graph and replace it with a path $uwv$, where $w$ is a new vertex. We say that $H$ is a topological minor of $G$ if $G$ contains a subdividion of $H$ as a subgraph. 

Let $H$ be a graph. In the $H$-subdivision game, Maker wins if $H$ is a topological minor of $M$. If $\Delta(H)\leq 3$, then $H$ is a topological minor of $M$ if and only if $H$ is a minor of $M$, so we will focus on the case $\Delta(H)\geq 4$. For this game, we will show that the maximum bias for which {\tt RandomMaker} and {\tt CleverMaker} win are within at most a factor of two from each other:

\begin{theo}\label{subdiv} Let $H$ be a graph with $\Delta(H)\geq 4$. In the $H$-subdivision game, {\tt CleverBreaker} wins against {\tt CleverMaker} for bias $b=\frac{2n}{\Delta(H)-1}$, while {\tt Random\-Maker} wins with high probability against {\tt CleverBreaker} for bias $b=\frac{(1-\epsilon)n}{\Delta(H)-1}$, for every fixed $\epsilon>0$. In particular, $b^*_{CC}\sim\frac{n}{\Delta(H)-1}\sim b^*_{RC}$.\end{theo}

This paper will be organized as follows: we will discuss most cases of the $H$-minor game and prove Theorem \ref{main} in Section \ref{section2}, we will discuss the separation between the {\tt RandomMaker} and {\tt CleverMaker} settings and prove Theorem \ref{pathrmcb} and Corollary \ref{bedlucnot} in Section \ref{section3}, and discuss the $H$-subdivision game and prove Theorem \ref{subdiv} in Section \ref{section4}. 

\subsection{Notation}

We say ``at time $t$'' for ``after $t$ moves by each player''. We denote by $M_t$, $B_t$ and $F_t$ the graph formed by the edges claimed by Maker, claimed by Breaker and which remain free, respectively, at time $t$. We denote by $e_t$ the edge claimed by Maker on round $t$.

\section{Proof of Theorem \ref{main}}\label{section2}

We say that a graph is in Class $i$, $1\leq i\leq 4$, if it falls under the $i$-th case of Theorem \ref{main}. We will start by proving the theorem for Classes 1, 2 and 4.

\subsection{Class 1}

A graph on Class 1 consists of isolated vertices and isolated edges. A graph $G$ on $n$ vertices contains $H$ as a minor if and only if $n\geq v(H)$ and $G$ contains a matching of size $e(H)$. For $n$ large enough, we only need to check the matching condition.

Note that the number of edges in Maker's graph at the end of the game is $\left\lceil\frac{{n\choose 2}}{b+1}\right\rceil$, so if $b\geq \frac{n^2}{2e(H)-2}$ then Maker claims fewer than $e(H)$ edges, and therefore loses regardless of the strategy followed. In particular, {\tt CleverBreaker} wins against {\tt CleverMaker}.

If the game lasts for at least $t$ rounds, the probability that $e_t$ shares a vertex with one of the previous edges is at most $\frac{2(t-1)n}{{n \choose 2}-(b+1)(t-1)}$, since the previous edges involve at most $2(t-1)$ vertices. The probability that the first $t$ edges form a matching is at least \[\prod\limits_{i=1}^t\left(1-\frac{2(i-1)n}{{n\choose 2}-(b+1)(i-1)}\right)\geq\left(1-\frac{2(t-1)n}{{n \choose 2}-(b+1)(t-1)}\right)^t,\] which is $1-o(1)$ for $t=e(H)$ and $b=\frac{(1-\epsilon)n^2}{2e(H)-2}$ for any $\epsilon>0$.

\subsection{Class 2}

A graph on Class 2 consists of isolated vertices and paths of length 1 and 2. In particular, such a graph has $\Delta(H)=2$, and every graph $G$ with $H$ as a minor has $\Delta(G)\geq 2$. For bias $b=2n$, {\tt CleverBreaker} can prevent {\tt CleverMaker} from creating a vertex of degree at least 2. More generally, the following lemma allows {\tt CleverBreaker} to bound $\Delta(M)$:

\begin{lemma}\label{deglem} Let $k\geq 2$ be an integer. For bias $b=\frac{2n}{k-1}$, {\tt CleverBreaker} has a strategy to ensure $\Delta(M)<k$.\end{lemma}

\begin{proof} Let $e_t=uv$. {\tt Clever\-Breaker} claims $\left\lfloor\frac{n}{k-1}\right\rfloor$ free edges incident to each of $u$ and $v$ (if $d_{F_{t-1}}(u)\leq \left\lfloor\frac{n}{k-1}\right\rfloor$ or $d_{F_{t-1}}(v)\leq \left\lfloor\frac{n}{k-1}\right\rfloor$, {\tt Clever\-Breaker} claims all free edges). Following this strategy, we maintain the invariant $d_{B_t}(v)\geq\min\left\{d_{M_t}(v)\left\lfloor\frac{n}{k-1}\right\rfloor, n-1-d_{M_t}(v)\right\}$. If $d_{M_t}(v)=k-1$ we have $d_{M_t}(v)\left\lfloor\frac{n}{k-1}\right\rfloor\geq n-d_{M_t}(v)$, so $d_{M_t}(v)+d_{B_t}(v)=n-1$ and there is no free edge incident to $v$. Thus $d_{M_t}(v)$ can never reach $k$.
\end{proof}

We will now show that, for $b=(2-\epsilon)n$, {\tt RandomMaker} wins whp against {\tt CleverBreaker}. In fact, we will show that in the first $n^{2/3}$ rounds of the game  {\tt RandomMaker} manages to create $k$ disjoint paths of length two, which in particular contain $H$ as a subgraph if $k\geq v(H)$.

We define inductively the sets of edges $C_t$ and $D_t$, with $C_0$ and $D_0$ being empty, and then:

\begin{itemize}
\item If $e_{t+1}$ is vertex-disjoint with both $C_t$ and $D_t$, then set $C_{t+1}=C_t\cup\{e_{t+1}\}$, $D_{t+1}=D_t$.
\item If $e_{t+1}$ is vertex-disjoint with $D_t$ but not with $C_t$, let $e_j\in C_t$ be an edge sharing a vertex with $e_{t+1}$. Obtain $C_{t+1}$ from $C_t$ by removing every edge sharing a vertex with $e_{t+1}$, and let $D_{t+1}=D_i\cup\{e_j, e_{t+1}\}$.
\item If $e_{t+1}$ is not vertex disjoint with $D_t$, then set $C_{t+1}=C_t$ and $D_{t+1}=D_t$.
\end{itemize}

We can observe by induction that $C_t$ and $D_t$ are vertex-disjoint, that $C_t$ forms a family of vertex-disjoint edges, and $D_t$ forms a family of vertex-disjoint paths of length two. Our goal is to prove that, with high probability, $\left|D_{n^{2/3}}\right|\geq 2k$.

Observe that $|D_t|$ is non-decreasing. If $\left|D_{n^{2/3}}\right|<2k$ then $|D_t|<2k$ for every $0\leq t\leq n^{2/3}$. At time $t$ there are at least ${n \choose 2}-(b+1)n^{2/3}=\Omega(n^2)$ free edges from which {\tt RandomMaker} chooses randomly, and at most $3kn$ of them are not vertex-disjoint with $D_i$, meaning that the probability that {\tt RandomMaker} picks one of them is at most $O(n^{-1})$. With high probability, the edge $e_{t+1}$ is vertex-disjoint with $D_t$ for every $1\leq t\leq n^{2/3}$. We denote this event by $X$.

The size of $C_t$ decreases at most $k$ times, and on the rounds where it does we have $|C_t|-|C_{t+1}|\leq 2$. If $X$ holds then the size of $C_t$ does not remain constant, meaning that if in addition we have $|D_t|<2k$ then $|C_t|\geq t-3k$ for every $1\leq t\leq n^{2/3}$. 

The number of edges with one endpoint incident to $C_t$ and the other outside of both $C_t$ and $D_t$ is at least $(2t-3k)(n-2t)=(2-o(1))tn$, out of which both players have claimed at most $\left(2-\frac\epsilon 2\right)tn$, and thus at least $\frac{\epsilon}{4}tn$ of which are free. On each round $\frac12n^{2/3}\leq t\leq n^{2/3}$, the probability that {\tt RandomMaker} chooses one of these free edges is at least $\frac{\frac \epsilon4tn}{{n \choose 2}-2tn}\geq \frac{\epsilon}{8}n^{-1/3}$. By Chernoff's bound, the probability that at most $k$ times one of these edges is selected is at most $e^{-\Omega\left(m^{1/3}\right)}$. With high probability, this is not the case, so with high probability $|D_{n^{2/3}}|>3k$.

\subsection{Class 4}

The upper bound for this class is implied by a result of Bednarska and Pikhurko \cite{BedPik}, and the lower bound follows a result by Krivelevich \cite{KriPre}, in which a much larger minor is created:

\begin{theo} For bias $b=\lfloor\frac{n-1}2\rfloor$, {\tt CleverBreaker} has a strategy to prevent {\tt CleverMaker} from claiming a cycle.\end{theo}

\begin{theo}\label{kritheo} For every $\epsilon>0$ there exists $c>0$ with the following property: for bias $b=(1-\epsilon)\frac n2$, {\tt RandomMaker} whp creates a $K_{c\sqrt{n}}$ minor when playing against {\tt CleverBreaker}.\end{theo}

\subsection{Class 3}\label{subsc3}

Class 3 seems to be the most complicated case to analyze. We will first prove that $\frac n2$ is a sharp threshold in the game between {\tt CleverMaker} and {\tt CleverBreaker}, and then prove in the next section that the situation changes if {\tt RandomMaker} plays instead.

For the upper bound of the Clever-Clever case, note that every graph $H$ in Class 3 must contain either $K_{1,3}$ or $P_4$ as a subgraph. Note also that, if $G$ is a graph with $H$ as a minor, then $K_{1,3}\subseteq H\Rightarrow K_{1,3}\subseteq G$ and $P_4\subseteq H\Rightarrow P_4\subseteq G$. Therefore, {\tt CleverBreaker} only needs to be able to prevent {\tt CleverMaker} from creating a copy of $K_{1,3}$ or $P_4$ to win.

Suppose first that $K_{1,3}\subseteq H$. Then, for $b=n$, {\tt CleverBreaker} can follow the strategy from Lemma \ref{deglem} to win.

Suppose now that $P_4\subseteq H$. We will define a strategy for {\tt CleverBreaker} that forces every component in {\tt CleverMaker}'s graph to be a star. In this strategy, the following properties are true before every move by {\tt CleverMaker}:
\begin{inva}\label{starinv} Every componenet in $M_t$ has a distinguished vertex, called its head. The component is a star centered on its head. Every edge which is free at time $t$ connects two heads, at least one of which is an isolated vertex in $M_t$.\end{inva}

At the start of the game, every vertex is isolated, and it is the head of its component. Every time that {\tt CleverMaker} claims an edge $uv$, both endpoints are the heads of their corresponding components, and at least one of them, say $u$, is an isolated vertex. {\tt CleverBreaker} makes $v$ the head of the new component, claims every free edge incident to $u$ and claims every edge from $v$ to a non-isolated vertex.

We can check that Invariant \ref{starinv} holds at the start of the game and, by following the strategy, if the invariant holds at time $t$ then it also holds at time $t+1$. We will prove that the strategy can be followed with $b=n$. Observe that in each round the number of heads decreases by one, so after {\tt RandomMaker}'s $t$-th move there are $n-t$ heads, and hence at most $n-t$ free edges incident to $u$. On the other hand, since Maker has claimed $t$ edges, there are at most $t$ non-trivial components, hence at most $t$ non-isolated heads whose free edges to $v$ must be claimed by {\tt CleverBreaker}. This gives a total of $n$ edges at most that {\tt CleverBreaker} must claim to follow this strategy.

Next we show that {\tt CleverMaker} has a strategy to claim a graph with $H$ as a minor, for bias $b=(1-\epsilon)n$. This strategy is based on the box game, introduced by Chv\'atal and Erd\H{o}s  \cite{ChvErd}. This is a particular case of Maker-Breaker game, although we will study it from a different perspective.

Suppose that we have $k$ boxes, and let $a_i$ be the initial number of coins in the $i$-th box. Let $m$ be a positive integer, which will be the bias of the game. Two players, {\tt BoxMaker} and {\tt BoxBreaker}, play in this game. On every round, {\tt BoxMaker} removes a total of $m$ coins from one or more of the boxes. Then, {\tt BoxBreaker} removes a non-empty box. The game ends when one of the boxes becomes empty (in which case {\tt BoxMaker} wins) or when there are no boxes left (in which case {\tt BoxBreaker} wins).

Let $h_i=\sum\limits_{j=1}^ij^{-1}$ be the $i$-th harmonic number. One can give the following criterion that guarantees that {\tt BoxBreaker} has a winning strategy (see Corollary 5.4 in \cite{HamVer} for a tight version): for any non-empty set $S\subseteq [k]$, we have \begin{equation}\label{boxinv}\sum\limits_{j\in S}a_j> |S|h_{|S|}m.\end{equation} {\tt BoxBreaker}'s strategy is then to always remove the box with the fewest coins in it. The reason why this criterion is valid is that the number of coins in each box after any number of rounds still satisfies equation \eqref{boxinv}.

Now consider a variant of the previous game. Instead of removing a box, {\tt BoxBreaker} places the $m$ coins that {\tt BoxMaker} removed back into one of the boxes (all coins must go into a single box). {\tt BoxBreaker} cannot win (there is no winning condition in this variant), but under the same criterion as before, {\tt BoxBreaker} can indefinitely prevent {\tt BoxMaker} from winning.

With this in mind, we can give a sketch of the proof of the lower bound. Remember that {\tt CleverMaker}'s goal is to create a graph with $H$ as a minor, by creating a family of branch sets which are connected in the right way. During the game, {\tt CleverMaker} creates one by one the branch sets. The free edges incident to the branch sets play the role of the coins. {\tt CleverMaker} plays as {\tt BoxBreaker} in a parallel game. {\tt CleverBreaker}, as {\tt BoxMaker}, claims on every round some of the free edges (coins), removing them. {\tt CleverMaker} finds the branch set $C_i$ with the fewest free edges incident to it and claims appropriately one of them. The new edge connects $C_i$ to a vertex that was previously not included in any branch set, so the vertex can be added to $C_i$. This increases the number of free edges incident to $C_i$ (as expected in {\tt BoxBreaker}'s move).

While playing according to this strategy increases the size of the branch sets created, it does not increase the number of them. Therefore, it will be necessary to use some rounds to create a new branch set. This is possible because the number of free edges added in {\tt CleverMaker}'s turn is slightly larger than the number of edges claimed in {\tt CleverBreaker}'s turn.

\begin{theo}\label{boxtheo} Let $H$ be a forest. For bias $b=(1-\epsilon)n$, and $n>N(H,\epsilon)$, {\tt CleverMaker} has a strategy to win the $H$-minor game. \end{theo}

Now we describe the strategy in detail. Let $k=|V(H)|$, and let $v_1, \dots, v_k$ be a 1-degenerate ordering of $V(H)$ (every vertex is adjacent to at most one of its predecessors). Let $r(t)$ denote the number of branch sets created by the $t$-th round. Let $C_i(t)$ denote the branch set corresponding to $v_i$ at time $t$, and $f_i(t)$ the number of edges which are free before {\tt CleverBreaker}'s move on round $t$ and which are incident to $C_i(t)$.

\begin{itemize}
\item On the first round of the game, claim an arbitrary edge $uv$. Set $r(1)=1$, $C_1(1)=\{u,v\}$.
\item Let $f'_i(t-1)$ denote the number of free edges incident to $C_i(t-1)$ at time $t-1$. If for every non-empty subset $S\subseteq [r(t-1)]$ we have \begin{equation}\label{gboxcond} \sum\limits_{i\in S}f'_i(t-1)>h_{|S|+1}(|S|+1)\left(1-\frac\epsilon2\right)n\end{equation} then set $r(t)=r=r(t-1)+1$. Let $1\leq i<r$ be the only value such that $v_iv_r\in E(H)$ (if no such value of $i$ exists, select $i$ arbitrarily). Choose a free edge $uv$ with $u\in C_i(t-1)$, $v\not\in\cup_{j=1}^{r-1}C_j(t-1)$, and which maximizes $d_{F_{t-1}}(v)$. Set $C_r=\{v\}$.
\item If \eqref{gboxcond} does not hold for some $S$, let $i$ be the value that minimizes $f'_i(t-1)$. Choose a free edge $uv$ with $u\in C_i(t-1)$, $v\not\in\cup_{j=1}^{r(t-1)}C_j(t-1)$ and which maximizes $d_{F_{t-1}}(v)$. Set $C_i(t)=C_i(t-1)\cup\{v\}$.
\end{itemize}

We will see that, not only does this strategy lead to {\tt CleverMaker}'s victory, but that the duration of the game is bounded by a constant $K(H,\epsilon)$. After fixing such $K$, there exists $N(H,\epsilon, K)$ such that every inequality in this proof holds for $n>N$.

Let $T(t)$ be the minimum value of $s$ such that $r(s)=r(t)$. Let us study how the numbers $f_i(t)$ change as long as $r(t)$ remains constant. We claim that the following condition holds: for every $t\leq K$, for every non-empty $S\subseteq [r(t)]$, we have \begin{equation}\label{gboxinv}\sum\limits_{j\in S}f_i(t)\geq h_{|S|}|S|\left(1-\frac\epsilon2\right)n+\frac{|S|}{k}\left(t-T(t)\right)\frac\epsilon 4n.\end{equation}

Notice that {\tt CleverBreaker} claims $(1-\epsilon)n$ edges at each round, out of which at most ${2K \choose 2}\leq \frac{\epsilon n}{1000}$ have both endpoints in $\cup_{j=1}^{r(t-1)}C_j(t-1)$. Therefore we have \begin{equation}\label{maxdec}\sum\limits_{j=1}^{r(t-1)}\left(f_j(t-1)-f'_j(t-1)\right)\leq \left(1-\frac{999\epsilon}{1000}\right)n.\end{equation}

{\tt CleverMaker} and {\tt CleverBreaker} together claim fewer than $K(1-\frac{\epsilon}{2})n$ edges throughout the first $K$ rounds, meaning that among any $\frac{\epsilon}{1000K}n$ vertices there is one with at least $(1-\frac{\epsilon}{1000})n$ free edges incident to it. Assuming that \eqref{gboxinv} holds for $S=\{i\}$, and taking \eqref{maxdec} into account, we have that $f'_i(t)\geq \frac{\epsilon}{4}n$. There are at most ${2K \choose 2}<\frac{\epsilon}{20}n$ free edges with both endpoints in $\cup_{j=1}^{r(t-1)}C_j(t-1)$, so there are at least $\frac{\epsilon}5n$ free edges from $C_i(t-1)$ to outside $\cup_{j=1}^{r(t-1)}C_j(t-1)$, and the number of vertices outside $\cup_{j=1}^{r(t-1)}C_j(t-1)$ which have a neighbour in $C_i(t-1)$ is greater than $\frac{\epsilon}{1000K}n$, meaning that one of them is incident to at least $\left(1-\frac{\epsilon}{1000}\right)n$ free edges. By the choice of the edge selected by {\tt CleverMaker}, we have \begin{equation}\label{mininc} f_i(t)-f'_i(t-1)\geq \left(1-\frac\epsilon{500}\right)n \hskip 1cm\text{if }r(t-1)=r(t)\end{equation}\begin{equation}\label{minnew} f_{r(t)}(t)\geq \left(1-\frac\epsilon{500}\right)n \hskip 2.4cm \text{if }r(t-1)\neq r(t)  \end{equation}

We consider several cases:

Suppose first that \eqref{gboxcond} does not hold for some set $S$. Therefore {\tt CleverMaker} adds a vertex to a branch set $C_i$, and $r(t)=r(t-1)=r$. We will show that, if \eqref{gboxinv} holds for every non-empty $S\subseteq [r]$ replacing $t$ with $t-1$, then \eqref{gboxinv} itself also holds. Fix a set $S\subseteq [r]$. We distinguish two cases:

\begin{itemize}
\item \textbf{Case 1.1:} $i\in S$. In this case, combining \eqref{maxdec} and \eqref{mininc}, we get \begin{align*}\sum\limits_{j\in S}f_j(t)&\geq \sum\limits_{j\in S}f_j(t-1)+\frac{\epsilon}{4}n\\ &\geq h_{|S|}|S|\left(1-\frac\epsilon 2\right)n+\frac{|S|}{k}(t-T(t))\frac\epsilon 4.\end{align*}

\item \textbf{Case 1.2:} $i\notin S$. In this case, we consider \eqref{gboxinv} applied to the set $S\cup\{i\}$. Remember that, by the choice of $i$, we have $f'_i(t-1)\leq f'_j(t-1)$ for any $j\in S$. This leads to: \begin{align*}\sum\limits_{j\in S}f_j(t)&=\sum\limits_{j\in S}f'_j(t-1)\\ 
&\geq\frac{|S|}{|S|+1}\left(\sum\limits_{j\in S\cup\{i\}}f'_j(t-1)\right)\\
&\geq \frac{|S|}{|S|+1}\left(\sum\limits_{j\in S\cup\{i\}}f_j(t-1)-\left(1-\frac{999\epsilon}{1000}\right)n\right)\\
&\geq h_{|S|+1}|S|\left(1-\frac\epsilon 2\right)n+\frac{|S|}{k}(t-1-T(t))\frac\epsilon 4n-\frac{|S|}{|S|+1}\left(1-\frac{999\epsilon}{1000}\right)n\\
&=h_{|S|}|S|\left(1-\frac\epsilon 2\right)n+\frac{|S|}{k}(t-1-T(t))\frac\epsilon 4n+\frac{|S|}{|S|+1}\frac{499\epsilon}{1000}n\\
&\geq h_{|S|}|S|\left(1-\frac\epsilon 2\right)n+\frac{|S|}{k}(t-T(t))\frac\epsilon 4n.
\end{align*}
\end{itemize}

Now we consider the case in which $r(t-1)\neq r(t)=r$. This means that \eqref{gboxcond} holds for every non-empty $S\subseteq [r-1]$. Also, we have $T(t)=t$. We will show that $\eqref{gboxinv}$ also holds. We distinguish two cases:

\begin{itemize}
\item \textbf{Case 2.1:} $r\notin S$. This case follows directly from \eqref{gboxcond}:
\begin{align*}
\sum\limits_{j\in S}f_j(t)&=\sum\limits_{j\in S}f'_j(t-1)\\
&\geq h_{|S|+1}(|S|+1)\left(1-\frac\epsilon 2\right)n\\
&\geq h_{|S|}|S|\left(1-\frac\epsilon 2\right)n
\end{align*}

\item\textbf{Case 2.2:} $r\in S, S\neq \{r\}$. This case also follows from \eqref{gboxcond}:
\[\sum\limits_{j\in S}f_j(t)\geq\sum\limits_{j\in S\setminus \{r\}}f'_j(t-1)\geq h_{|S|}|S|\left(1-\frac\epsilon 2\right)n.\]

\item\textbf{Case 2.3:} $S=\{r\}$. This case follows from \eqref{minnew}:

\[f_r(t)\geq \left(1-\frac{\epsilon}{500}\right)n
\geq h_{1}\left(1-\frac\epsilon 2\right)n.\]
\end{itemize}

This concludes the proof of the invariant \eqref{gboxinv} for the first $K$ rounds. To complete the proof of Theorem \ref{boxtheo}, we need to show that for an appropriate value of $K=K(H,\epsilon)$ the value of $r(t)$ reaches $k$ before $t=K$. Indeed, by \eqref{maxdec}, if we have \begin{equation*}\sum\limits_{j\in S}f_j(t-1)\geq h_{|S|+1}(|S|+1)\left(1-\frac\epsilon 2\right)n+n\end{equation*} then \eqref{gboxcond} will also hold. But there is a constant $K'$, depending only on $\epsilon$ and $k$, such that \begin{equation*}h_{s}s\left(1-\frac\epsilon 2\right)n+\frac skK'\frac\epsilon 2n\geq h_{s+1}(s+1)\left(1-\frac\epsilon 2\right)n+n\end{equation*} holds for every $1\leq s\leq k$. Thus it takes at most $K'$ rounds to increase the value of $r(t)$, and $K=kK'$ is enough to satisfy the desired conditions. This concludes the proof of Theorem \ref{boxtheo}.

To conclude the proof of Theorem \ref{main}, simply observe that by Theorem \ref{kritheo}, {\tt RandomMaker} wins against {\tt CleverBreaker} for bias $b=(1-\epsilon)\frac n2$, which combined with the upper bound of the Clever-Clever case gives $b^*_{RC}\sim n$.

\section{Proof of Theorem \ref{pathrmcb}}\label{section3}

We will next prove Theorem \ref{pathrmcb}, implying that {\tt RandomMaker} cannot match the threshold of the game between {\tt CleverMaker} and {\tt CleverBreaker} for every graph $H$ in Class 3. For this, we will give a strategy for {\tt CleverBreaker} for bias $b=0.99n$.

The game lasts for $\frac{{n \choose 2}}{0.99n+1}=\frac{n}{1.98}+o(n)$ rounds. {\tt CleverBreaker}'s strategy will divide the game into two phases, one consisting of the first $0.03n$ rounds and the other consisting of the rest. In the first phase, {\tt CleverBreaker}'s goal is to claim linearly many edges incident to every vertex, while preventing {\tt RandomMaker} from creating a large component.  Let $I$ denote the set of isolated vertices in $M_{0.03n}$. If the first phase is successful, {\tt CleverBreaker}'s goal in the second phase is to make sure that {\tt RandomMaker}'s graph on $I$ consists of stars (following a strategy similar to the one from section \ref{subsc3}) and that every star in $I$ has at most one edge to $J=V\setminus I$.

Let us first analyze the first phase in more detail. {\tt CleverBreaker}'s strategy goes as follows. Let $I_t$ denote the set of isolated vertices in $M_t$, and let $J_t=V\setminus I_t$. {\tt CleverBreaker} claims $0.99n$ edges according to the following rules of priority:

\begin{itemize}
\item Edges within $J_t$, in any order.
\item Edges between $I_t$ and $J_t$, following this procedure: the endpoint $u$ in $J_t$ is chosen among the vertices of $J_t$ that are incident to a free edge, maximizing the size of its component in $M_t$ (ties are broken uniformly at random). The endpoint $v$ in $I_t$ is chosen uniformly at random from the free edges incident to $u$.
\item Edges within $I_t$, uniformly at random.
\end{itemize}

We claim that after $0.03n$ rounds the following two properties hold whp:
\begin{itemize}
\item For any $v_1,v_2\in I$, we have $|d_{F_{0.03n}}(v_1)-d_{F_{0.03n}}(v_2)|\leq 0.001n$.
\item Every component in $M_{0.03n}$ contains at most four vertices.
\end{itemize}

To prove the first claim, fix two vertices $v_1,v_2\in I$. We will bound the probability of $|d_{F_t}(v_1)-d_{F_t}(v_2)|\geq 0.001n$, for every $0\leq t\leq 0.03n$. 

Let $W_t(v_1,v_2)=\{w_1, w_2, \dots, w_k\}$ be the set of vertices in $V\setminus \{v_1,v_2\}$ which are joined by a free edge to exactly one of $\{v_1, v_2\}$ at time $t$. Since $v_1,v_2\in I$, the other edge was claimed by {\tt CleverBreaker}. We consider for each $i$ the indicating random variable of the event $X_i=\{w_iv_1\text{ is free}\}$. We claim that the random variables $X_i$ are mutually independent and each has probability $\frac12$. Indeed, consider a possible development $\mathcal{D}=\{M_j=\Gamma_j, B_j=\Psi_j\}_{j=0}^t$ of the game up to round $t$, in which the edge $w_iv_1$ was claimed at round $t_i$. Consider an alternate development $\mathcal{D}'=\{M_j=\Gamma_j, B_j=\Psi'_j\}_{j=0}^t$, in which $\Psi'_j=\Psi_j+w_iv_2-w_iv_1$ for $t_i\leq j\leq t$. If $\mathcal{D}$ satisfies the rules for each player then so does $\mathcal{D'}$, and the probability of reaching $\mathcal{D}$ and $\mathcal{D}'$ is the same (the relevant choices by either player are taken uniformly at random). Also note that the set $W_t(v_1,v_2)$ is the same in each development, with the only difference being the vertex to which $v_i$ is connected. The mutual independence of the $\{X_i\}_{i=1}^k$ follows.

We have $|d_{F_t}(v_1)-d_{F_t}(v_2)|=\left|2\sum_{i=1}^kX_i-k\right|$. Conditioned on $k$, the distribution of $\sum_{i=1}^kX_i$ is a binomial distribution $\mathrm{Bin}\left(k,\frac12\right)$. By Chernoff's inequality, $\Pr(|d_{F_t}(v_1)-d_{F_t}(v_2)|\geq 0.001n\mid k)\leq 2\exp\left(-0.0005^2n\right)$, which is exponentially small. We conclude that, with high probability, $|d_{F_t}(v_1)-d_{F_t}(v_2)|\leq 0.001n$ for all $0\leq t\leq 0.03n$ and every $v_1,v_2\in I$.

Now we will prove the second claim. Observe first that on every round at most two vertices are added to $J_t$, and so $|J_t|\leq 2t$. In addition, at most $2|J_t|\leq 0.12n$ free edges are added to $J_t$, and since those edges have the highest priority, {\tt CleverBreaker} claims all of them in one round. Therefore, after {\tt CleverBreaker}'s move there are no free edges inside $J_t$. Every edge claimed by {\tt RandomMaker} has at least one endpoint in $I_t$, meaning that in every round at least one vertex is added to $|J_t|$, and $|J_t|\geq t$.

Let $m_k(t)$ denote the number of free edges incident to components with at least $k$ vertices in $M_t$. As every edge claimed by {\tt RandomMaker} has one of its endpoints in $I_t$, we get that $m_k$ can only increase if {\tt RandomMaker} claims an edge incident to a previously created component of size exactly $k-1$ (in which case $m_k(t+1)\leq m_k(t)+kn-b$) or at least $k$ (in which case $m_k(t+1)\leq m_k(t)+n-b$). Otherwise, since the free edges incident to large components have priority in {\tt CleverBreaker}'s strategy, we have $m_k(t+1)\leq \max\{0,m_k(t)-b+0.12n\}$.

Following these inequalities, $m_2(t+1)\leq m_2(t)+1.01n$, and thus $m_2(t)\leq 0.0303n^2$. In the first $0.03n$ rounds the amount of edges claimed by the two players is at most $0.03n(b+1)\leq {n\choose 2}-0.47n^2$, therefore on each round the probability that {\tt RandomMaker} creates a component of size at least three is at most $\frac{0.0303n^2}{0.47n^2}<\frac1{15}$.

Next we claim that $m_3(t)\leq n^{4/3}$ for every $0\leq t\leq 0.03n$. Assume that $m_3(T)\geq n^{4/3}$ for some $T$, and let $T'\leq T$ be the largest value for which $m_3(T')=0$. In the rounds $T'\leq t\leq T$ we have $m_3(t)>0$, and so the value of $m_3(t)$ increases by at most $2.01n<3n$ when {\tt RandomMaker} creates a component of size at least three, and it decreases by at least $0.87n>\frac n2$ otherwise. From this we deduce that $T-T'\geq\frac{n^{1/3}}{3}$ and that, between the rounds $T'$ and $T$ there were at least $\frac{T-T'}{7}$ increases (as one increase can compensate at most six decreases). But an increase occurs with probability at most $\frac 1{15}$ on each round, regardless of the result of previous rounds, so for any fixed values of $T,T'$ with $T-T'\geq \frac{n^{1/3}}{3}$ the probability of $\frac{T-T'}{7}$ increases is, by Chernoff's inequality, at most $\exp\left(-\frac{T-T'}{3}\right)\leq \exp\left(-\frac{n^{1/3}}{9}\right)$. With high probability, this is not the case for any $0\leq T'<T\leq 0.03n$, and $m_3(t)\leq n^{4/3}$.

On each round, the probability that {\tt RandomMaker} creates a component of size at least four is at most $\frac{n^{4/3}}{0.47n^2}<3n^{-2/3}$. For every $T$, the probability that two components of size at least four are created between the rounds $T$ and $T+n^{1/10}$ is at most $9n^{-4/3}{n^{1/10} \choose 2}=o(n^{-1})$. With high probability, any two components of size at least four are created at least $n^{1/10}$ rounds apart.

Any component of size exactly four is incident to at most $4n$ free edges. We know that whp the component will not grow in the $n^{1/10}$ rounds after being created, and {\tt CleverBreaker} claims all the edges incident to it in the first $\left\lceil\frac{4n}{0.87n}\right\rceil=5$ rounds after its creation. We conclude that the components of size four do not grow at all, and that no component of size at least five is created.

Now we describe {\tt CleverBreaker}'s strategy during the second phase of the game ($t>0.03n$). Remember that there are no free edges in $J$, and that $|I|\leq 0.97n$. At any point in the game, we talk about the $I$-component or the $J$-component of a vertex as the component of that vertex in the graph induced by {\tt RandomMaker}'s edges on $I$ or $J$. Every $I$-component will have a head, as in the strategy to prevent paths with four vertices from section \ref{subsc3}. Initially every vertex in $I$ is its own head.

\begin{itemize}\item Whenever {\tt RandomMaker} claims an edge $uv$ within $I$, wlog $u$ is an isolated vertex. Then {\tt CleverBreaker} claims every free edge incident to $u$ within $I$, and every edge from $v$ to the heads of other non-trivial $I$-components. This requires at most $0.97n$ edges. Then claim free edges incident to the newly created $I$-component until $0.99n$ edges have been claimed, or until no free edges incident to the $I$-component are left.
\item Whenever {\tt RandomMaker} claims an edge $uv$ between $I$ and $J$, {\tt Clever\-Breaker} claims every free edge incident to the corresponding $I$-component.\end{itemize}

We claim that the following invariants hold:

\begin{inva}\label{invapha2}
\
\begin{itemize}\item Every $I$-component is a star, with the head as its center.\item Every free edge within $I$ connects two heads, at least one of which is an isolated vertex.\item Every $I$-component is incident to at most $0.98n$ free edges.\item Every $I$-component has at most one {\tt RandomMaker} edge connecting it to $J$, and if there is such an edge then the $I$-component is not incident to any free edges.\end{itemize}\end{inva}

The proof of the first two items is the same as in section \ref{subsc3}. We will prove the latter two. We start by showing that, at the start of the second phase, every vertex in $I$ is incident to at most $0.98n$ free edges. The number of edges claimed by {\tt CleverBreaker} during the first phase is at most $0.03n\cdot 0.99n=0.0297n^2$. Out of these edges, at most ${|J| \choose 2}$ edges lie inside $J$ and the rest have at least one endpoint in $I$, meaning that the sum of the degrees of the vertices in $I$ in {\tt CleverBreaker}'s graph is at least \[0.0297n^2-{|J| \choose 2}\geq 0.0297n^2-{0.06n \choose 2}>0.0279n^2\] The average degree of the vertices of $I$ in {\tt CleverBreaker}'s graph is at least $0.0279n$, so the average number of free edges incident to it is at most $0.9721n$. But the number of free edges incident to two vertices in $I$ can differ by at most $0.001n$, hence every vertex in $I$ is incident to at most $0.98n$ free edges.

Now suppose that Invariant \ref{invapha2} holds at time $t$. Let $e_{t+1}=uv$. If $u\in I$ and $v\in J$, then {\tt CleverBreaker} can claim every free edge incident to the $I$-component of $u$, as there are at most $0.98n$ of them, and the invariant holds at round $t+1$ (the $I$-components have not changed). If $u,v\in I$, the newly created $I$-component is incident to at most $1.96n$ free edges, out of which at most $0.97n$ will remain after {\tt CleverBreaker}'s move. The newly created $I$-component is also not adjacent to $J$, since both $u$ and $v$ were incident to a free edge in the previous round. Invariant \ref{invapha2} holds at time $t+1$.

The $J$-components all lie in different components in $M$, so a path in $M$ intersects at most one $J$-component. In addition, every $I$-component can only appear at the beginning or the end of a path, because at most one edge leaves each $I$-component. Moreover, a $J$-component can only contribute four vertices to a path (that is the maximum possible size of a $J$-component) and an $I$-component can only contribute three vertices (because it is a star). Thus the maximum possible number of vertices in a path in $M$ is ten, one short of the eleven needed for a path on ten edges. This concludes the proof of Theorem \ref{pathrmcb}.

As a corollary, observe that if $H$ is a path or a three-legged spider (three disjoint paths meeting at a common endpoint) then any graph that contains $H$ as a minor also contains $H$ as a subgraph. Therefore the constant factor separation that we observe in the $H$-minor game also translates to the $H$-subgraph game. This means that the result from \cite{BedLuc} showing $b^*_{CC}\sim b^*_{RC}$ in the $H$-subgraph game cannot be strengthened to $b^*_{CC}\approx b^*_{RC}$, as we stated in Corollary \ref{bedlucnot}.

\begin{figure}
\begin{centering}
\includegraphics[width=85mm]{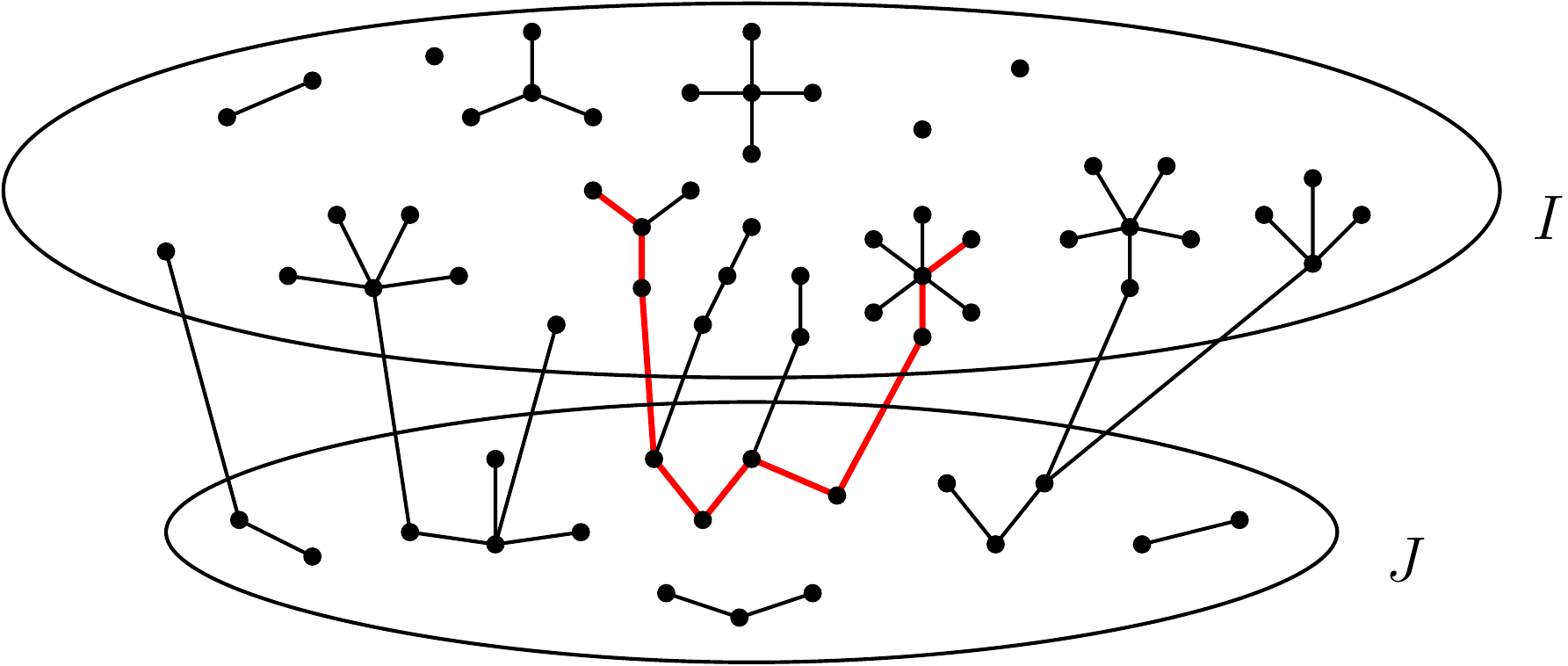}
\caption{A path of maximum length in a graph with the structure of $M$.}
\par\end{centering}
\end{figure}

\section{The $H$-subdivision game}\label{section4}

We will now prove Theorem \ref{subdiv}. For the upper bound simply observe that, if {\tt CleverMaker} wins, then $\Delta(M)\geq \Delta(H)$. Therefore, by Lemma \ref{deglem} with $k=\Delta(H)$, {\tt CleverBreaker} wins with bias $b=\frac{2n}{\Delta(H)-1}$.

For the lower bound, let $b=\frac{(1-\epsilon)n}{\Delta(H)-1}$. The key tool is a sufficient condition for the existence of $H$-subdivisions, given by Mader \cite{Mad}:
\begin{theo}\label{mader} Let $H$ be a graph with $\Delta(H)\geq 3$ and let $\delta>0$. There exists an integer $\ell(H,\delta)$ such that every graph with girth at least $\ell$ and average degree at least $\Delta(H)-1+\delta$ contains a subdivision of $H$.\end{theo}

The game lasts for $\left\lceil\frac{{n \choose 2}}{b+1}\right\rceil\geq \frac{\Delta(H)-1+4\delta}{2}n$ rounds for some $\delta=\delta(\epsilon)$, so this is the number of edges claimed by {\tt RandomMaker}. Let $\ell=\ell(H,\delta)$ be as in Theorem \ref{mader}. We will prove that whp {\tt RandomMaker}'s graph $M$ contains a subgraph $M'$ with girth at least $\ell$ and average degree at least $\Delta(H)-1+\delta$.

We define the subgraph $M'$ as follows: on every round, we add $e_t$ to $M'$ unless it is incident to a vertex that already has degree at least $\log n$ in $M'$ or it creates a cycle in $M'$ of length smaller than $\ell$. By construction, the girth of $M'$ is at least $\ell$. We will now prove that whp $e(M')\geq \frac{\Delta(H)-1+\delta}{2}n$.

On round $t$, with $t\leq\frac{\Delta(H)-1+2\delta}{2}n$, there are at least $\delta nb=\frac{\delta(1-\epsilon)n^2}{\Delta(H)-1}$ free edges. There are at most $\frac{\Delta(H)-1+2\delta}{\log n}n$ vertices of degree $\log n$ in $M_t$, so there are at most $\frac{\Delta(H)-1+2\delta}{\log n}n^2$ free edges incident to one of them.  On the other hand, there are at most $(\Delta(M'))^\ell\leq (\log n)^\ell$ vertices at distance at most $\ell$ from a given vertex, so the number of free edges that would close a cycle of length shorter than $\ell$ is at most $n(\log n)^\ell$. The probability that a randomly selected free edge is not added to $M'$ is at most $\frac{\frac{\Delta(H)-1+2\delta}{\log n}n^2+n(\log n)^\ell}{\frac{\delta (1-\epsilon) n^2}{\Delta(H)-1}}=o(1)$, and by Markov's inequality the probability that $\frac\delta 2n$ edges fail to be added to $M'$ in the first $\frac{\Delta(H)-1+2\delta}{2}n$ rounds is also $o(1)$.

We conclude that whp the subgraph $M'$ of {\tt RandomMaker}'s graph satisfies the conditions of Theorem \ref{mader} and so contains a subdivision of $H$.

\section{Concluding remarks}

In both the $H$-minor and the $H$-subdivision games, the random strategy for Breaker is close to being optimal. This appears to be the case in many different games in which the property $\mathcal{P}$ is hereditary increasing: if $G_1$ and $G_2$ are graph on $n'$ and $n$ vertices, respectively, with $n'\leq n$ and $G_1\subset G_2$, then $G_1\in \mathcal P$ implies $G_2\in\mathcal P$. Indeed, compare these two games and the $H$-subgraph game to the Hamiltonicity game, in which $b^*_{CC}\approx\frac{n}{\log n}$ but {\tt CleverBreaker} wins against {\tt RandomMaker} whp for $b=1$ (see \cite{GroSza}, \cite{KriKro}). It is not known whether for every hereditary increasing property $\mathcal{P}$ it is true that $b^*_{CC}\sim b^*_{RC}$.

The question of when does a game satisfy $b^*_{CC}\approx b^*_{RC}$ seems to be harder to answer. The $H$-minor games for $H$ in Class 1, Class 2 and Class 4 do not seem to fit into any general criterion for which this phenomenon holds, or can be conjectured to hold.

\section{Acknowledgements}

The content of this paper is based on my Master's thesis. I would like to thank my advisor, Tibor Szab\'o, for suggesting this topic and for the discussions that we had during the elaboration of the thesis, defense preparation and elaboration of this paper.

\bibliography{bibliog}{}
\bibliographystyle{abbrv}
~\\
~\\
\begin{minipage}[t]{0.5\linewidth}
	Ander Lamaison\\ 
	\texttt{<lamaison@zedat.fu-berlin.de>}\\
	Institut f\"ur Mathematik, Freie Universit\"at Berlin and Berlin Mathematical School, Berlin, Germany. %\\ 
\end{minipage}

\end{document}